\documentclass[10pt]{amsart}
\usepackage{enumerate,amsmath,amssymb,latexsym,
amsfonts, amsthm, amscd, MnSymbol}

\theoremstyle{plain}

\newtheorem{theorem}{Theorem}

\newtheorem*{theorem*}{Theorem}

\numberwithin{equation}{section}

\newcommand{\ii}{{\rm i}}
\newcommand{\dd}{{\rm dn}_2}

\newcommand{\oo}{\omega}
\newcommand{\OO}{\Omega}

\begin{document}

\title {Elliptic functions in signature four}

\date{}

\author[P.L. Robinson]{P.L. Robinson}

\address{Department of Mathematics \\ University of Florida \\ Gainesville FL 32611  USA }

\email[]{paulr@ufl.edu}

\subjclass{} \keywords{}

\begin{abstract}
We explore the relationships between two elliptic functions constructed by Shen in the signature four Ramanujan theory. 
\end{abstract}

\maketitle

\medbreak

Li-Chien Shen has introduced two separate elliptic functions in the Ramanujan theory of elliptic functions in signature four. In [2014] he constructed a partial analogue $\dd$ of the classical Jacobian elliptic function ${\rm dn}$, involving an incomplete integral of the hypergeometric function $F(\tfrac{1}{4}, \tfrac{3}{4}; \tfrac{1}{2}; \bullet)$ and a choice of modulus. In [2016] he constructed an elliptic function $y_4$ as the solution to a first-order initial value problem involving the Chebyshev polynomial $T_4$ and a choice of parameter. Here, we make explicit the relationship between these functions by making explicit the relationship between their coperiodic Weierstrass $\wp$ functions; as a by-product of our analysis, we identify the zeros and poles of the elliptic function $y_4$. 

\medbreak 

\section{The elliptic function $\dd$}

\medbreak 

The elliptic function $\dd$ is introduced in [2014] by means of an incomplete integral of the hypergeometric function $F(\tfrac{1}{4}, \tfrac{3}{4}; \tfrac{1}{2}; \bullet)$. We outline the construction here; for details, see [2014] and [2020]. 

\medbreak 

Fix $\kappa \in (0, 1)$ as modulus and $\lambda = \sqrt{1 - \kappa^2} \in (0, 1)$ as complementary modulus.

\medbreak 

The assignment 
$$T \mapsto \int_0^T F(\tfrac{1}{4}, \tfrac{3}{4}; \tfrac{1}{2}; \kappa^2 \sin^2 t) \, {\rm d} t$$ 
defines a strictly increasing bijection from $\mathbb{R}$ to itself, whose inverse we denote by $\phi: \mathbb{R} \to \mathbb{R}.$ Define $\oo$ by $\phi ( \oo) = \frac{1}{2} \pi$: thus 
$$\oo = \int_0^{\frac{1}{2} \pi} F(\tfrac{1}{4}, \tfrac{3}{4}; \tfrac{1}{2}; \kappa^2 \sin^2 t) \, {\rm d} t$$ 
and so 
$$\oo = \tfrac{1}{2} \pi F(\tfrac{1}{4}, \tfrac{3}{4}; 1; \kappa^2)$$
by a standard hypergeometric integral evaluation. An additional integral calculation shows that if $u \in \mathbb{R}$ then 
$$\phi (u + 2 \oo) = \phi(u) + \pi$$
whence the composite $\sin \phi : \mathbb{R} \to \mathbb{R}$ switches sign upon addition of $2 \oo$. Introduce the auxiliary function 
$$\psi = \arcsin (\kappa \sin \phi);$$ 
the cosine function being even, the composite $d: = \cos \psi : \mathbb{R} \to \mathbb{R}$ has period $2 \oo$. 

\medbreak 

Now $\dd$ is defined as the elliptic extension of $d$ to the plane. That $d$ so extends is an immediate consequence of its satisfying the following initial value problem. 

\medbreak 

\begin{theorem} \label{dn2}
The function $d = \cos \psi$ satisfies $d(0) = 1$ and the differential equation 
$$(d')^2 = 2 \, (1 - d) (d^2 - \lambda^2).$$
\end{theorem} 

\begin{proof} 
Aside from differentiation, the proof uses trigonometric duplication and the standard hypergeometric evaluation 
$$F(\tfrac{1}{4}, \tfrac{3}{4}; \tfrac{1}{2}; \sin^2 z) = \frac{\cos \frac{1}{2} z}{\cos z}$$
from page 101 of [1953]. For details, see [2014] or [2020]. 
\end{proof} 

\medbreak 

As announced, the fact that $d$ extends to an elliptic function $\dd$ is now in evidence: we can identify $\dd$ in terms of its coperiodic Weierstrass $\wp$ function, as follows.

\medbreak 

\begin{theorem} \label{p}
The elliptic function $\dd$ is given by 
$$(1 - \dd) (\tfrac{1}{3} + p) = \tfrac{1}{2} \kappa^2$$ 
where $p = \wp (\bullet; g_2, g_3)$ is the Weierstrass function with invariants 
$$g_2 = \tfrac{1}{3} (3 \lambda^2 + 1) \; \; {\rm and} \; \; g_3 = \tfrac{1}{27} (9 \lambda^2 - 1).$$  
\end{theorem} 

\begin{proof} 
A straightforward application of the Theorem in the Appendix. 
\end{proof} 

\medbreak 

\medbreak  

The invariants $g_2, g_3$ being real and the discriminant $g_2^3 - 27 g_3^2 = \kappa^4 \lambda^2$ being positive, the Weierstrass function $p = \wp (\bullet; g_2, g_3)$ has a fundamental pair of periods in which the first period is strictly positive and the second period is purely imaginary with strictly positive imaginary part: the first of these is precisely $2 \oo$; the second we shall denote by $2 \oo '$. It is a familiar property of such Weierstrass functions that as the perimeter of the half-period rectangle is traversed in the counterclockwise order 
$$0, \, \oo, \, \oo + \oo', \, \oo ', \, 0$$ 
the function $p$ strictly decreases through all real values, being critical at the nonzero vertices, the (midpoint) values there being in this case 
$$p(\oo) = \tfrac{1}{6} + \tfrac{1}{2} \lambda, \, p(\oo + \oo') = \tfrac{1}{6} - \tfrac{1}{2} \lambda, \, p(\oo') = - \tfrac{1}{3}.$$ 

\medbreak 

Theorem \ref{p} makes it clear that $p$ and $\dd$ are coperiodic; accordingly, we may use $\dd$ to derive expressions for $\oo$ (again) and $\oo '$. For this purpose, we observe that 
$$\dd (\oo) = \lambda \; \; {\rm and} \; \; \dd(\oo + \oo ') = - \lambda$$ 
while  $\dd$ has a pole at $\oo '$. 

\medbreak 

Define the acute modular angle $\alpha  \in (0, \tfrac{1}{2} \pi)$ by $\kappa = \sin \alpha$ so that also $\lambda = \cos \alpha$. 

\medbreak 

\begin{theorem} \label{ppitrig}
The real half-period $\oo$ of $\dd$ and $p$ is given by 
$$\oo = \sqrt2 \int_0^{\alpha} \frac{\cos \frac{1}{2} \theta}{\sqrt{\cos 2 \theta - \cos 2 \alpha}} \, {\rm d} \theta.$$
\end{theorem} 

\begin{proof} 
We integrate along the lower edge $[0, \oo]$ of the half-period rectangle. As $\dd$ decreases from value $1$ at $0$ to value $\lambda$ at $\oo$ it follows from the differential equation in Theorem \ref{dn2} that 
$$\dd ' = - \sqrt{2 (1 - \dd) (\dd^2 - \lambda^2)}$$
and 
$$\oo = \int_{\lambda}^1 \frac{1}{\sqrt{2 (1 - x) (x^2 - \lambda^2)}} \, {\rm d} x.$$
Here, substitute $x = \cos \theta$ for $0 \leqslant \theta \leqslant \alpha$; observe that  
$$\frac{{\rm d} x}{{\rm d} \theta} = - \sin \theta = - 2 \sin \tfrac{1}{2} \theta \cos \tfrac{1}{2} \theta$$ 
and 
$$1 - x = 1 - \cos \theta = 2 \sin^2 \tfrac{1}{2} \theta$$ 
while 
$$x^2 - \lambda^2 = \cos^2 \theta - \cos^2 \alpha = \tfrac{1}{2} (\cos 2 \theta - \cos 2 \alpha)$$
to complete the proof. 

\end{proof} 

\medbreak 

We may confirm that this value for $\oo$ agrees with its original definition: 
$$\oo = \int_0^{\frac{1}{2} \pi} F(\tfrac{1}{4}, \tfrac{3}{4}; \tfrac{1}{2}; \kappa^2 \sin^2 \phi) \, {\rm d} \phi$$ 
in which we have written $\phi$ for the integration variable. Change the integration variable to $\psi = \arcsin (\kappa \sin \phi)$: the hypergeometric evaluation that was employed in Theorem \ref{dn2} gives 
$$F(\tfrac{1}{4}, \tfrac{3}{4}; \tfrac{1}{2}; \kappa^2 \sin^2 \phi) = F(\tfrac{1}{4}, \tfrac{3}{4}; \tfrac{1}{2}; \sin^2 \psi) = \frac{\cos \frac{1}{2} \psi}{\cos \psi}$$
while the chain rule gives 
$$\frac{{\rm d} \phi}{{\rm d} \psi} = \frac{\cos \psi}{\kappa \cos \phi} = \frac{\cos \psi}{\sqrt{\kappa^2 - \sin^2 \psi}} = \frac{\cos \psi}{\sqrt{\cos^2 \psi - \cos^2 \alpha}} = \sqrt2 \frac{\cos \psi}{\sqrt{\cos 2 \psi - \cos 2 \alpha}}.$$

\medbreak 

It is convenient to record the following evaluation. 

\medbreak 

\begin{theorem} \label{eval}
$$\int_0^{\alpha} \frac{\cos \frac{1}{2} \theta}{\sqrt{\cos 2 \theta - \cos 2 \alpha}} \, {\rm d} \theta = \tfrac{1}{2 \sqrt2} \pi F(\tfrac{1}{4}, \tfrac{3}{4}; 1 ; \sin^2 \alpha).$$ 
\end{theorem} 

\begin{proof} 
Immediate from Theorem \ref{ppitrig} and the earlier standard hypergeometric integral evaluation.
\end{proof} 

\medbreak 

Write $\beta = \frac{1}{2} - \alpha$ for the complementary modular angle: so $\sin \beta = \lambda$ and $\cos \beta = \kappa$.  

\medbreak 

\begin{theorem} \label{ppi'}
The imaginary half-period $\oo '$ of $\dd$ and $p$ is given by 
$$\oo ' = 2 \ii \int_0^{\beta} \frac{\cos \frac{1}{2} \theta}{\sqrt{\cos 2 \theta - \cos 2 \beta}} \, {\rm d} \theta.$$
\end{theorem} 

\begin{proof} 
Integrate up the right edge $[\oo, \oo + \oo ']$ of the half-period rectangle: for $0 \leqslant t \leqslant - \ii \, \oo '$ let $g(t) = \dd (\oo + \ii t)$. Then 
$$(g ')^2 = - 2 (1 - g)(g^2 - \lambda^2)$$
and $g$ decreases from $g(0) = \lambda$ to $g(- \ii \, \oo ') = - \lambda$ so that 
$$ g ' = - \sqrt{2 (1 - g)(\lambda^2 - g^2)}$$ 
and therefore 
$$- \ii \, \oo ' = \int_{- \lambda}^{\lambda} \frac{1}{\sqrt{2 (1 - x)(\lambda^2 - x^2)}} \, {\rm d} x.$$
Substitute $x = \cos t$ for $\alpha \leqslant t \leqslant \pi - \alpha$ and again use trigonometric duplication to see that 
$$- \ii \, \oo ' = \sqrt2 \, \int_{\alpha}^{\pi - \alpha} \frac{\cos \frac{1}{2} t}{\sqrt{\cos 2 \alpha - \cos 2 t}} \, {\rm d} t.$$
Substitute $t = \tfrac{1}{2} \pi - \theta$ and use 
$\cos \tfrac{1}{2} t = \tfrac{1}{\sqrt2} (\cos \tfrac{1}{2} \theta + \sin \tfrac{1}{2} \theta)$ along with trigonometric parity to arrive at 
$$- \ii \, \oo ' = 2  \int_0^{\beta} \frac{\cos \frac{1}{2} \theta}{\sqrt{\cos 2 \theta - \cos 2 \beta}} \, {\rm d} \theta.$$
\end{proof} 

\medbreak 

Taken together, Theorem \ref{eval} and Theorem \ref{ppi'} show that 
$$\oo ' = \tfrac{1}{\sqrt2} \pi F(\tfrac{1}{4}, \tfrac{3}{4}; 1 ; \lambda^2).$$ 

\medbreak 

We now see the shape of the period rectangle: the period ratio is given by 
$$\frac{\oo '}{\oo} = \ii \, \sqrt2 \; \frac{F(\tfrac{1}{4}, \tfrac{3}{4}; 1 ; \lambda^2)}{F(\tfrac{1}{4}, \tfrac{3}{4}; 1; \kappa^2)} = \ii \, \sqrt2 \; \frac{F(\tfrac{1}{4}, \tfrac{3}{4}; 1 ; 1 - \kappa^2)}{F(\tfrac{1}{4}, \tfrac{3}{4}; 1; \kappa^2)}.$$

\medbreak 

\section{The elliptic function $y_4$}

\medbreak 

The elliptic function $y_4$ is introduced in [2016] as the solution to a specific initial value problem involving the degree-four Chebyshev polynomial (`of the first kind'); recall that this polynomial $T_4$ is given by $T_4 (t) = 8 t^4 - 8 t^2 + 1$ and derives its significance from the property $T_4 (\cos \theta) = \cos 4 \theta$. 

\medbreak 

To be explicit, we shall consider the differential equation 
$$(y ')^2 = T_4 (y) - (1 - 2 \lambda^2) = 8 y^4 - 8 y^2 + 2 \lambda^2$$ 
in which we have deliberately taken the parameter $\mu$ of [2016] to be the modulus $\lambda = \sqrt{1 - \kappa^2}$ that is complementary to $\kappa \in (0, 1)$. The quartic on the right-hand side of this equation has four distinct real zeros: namely 
$$\mu^{\pm} = \sqrt{\tfrac{1}{2} (1 \pm \kappa)} \; \; {\rm and} \; \; - \mu^{\pm}.$$
We augment the differential equation by imposing the initial condition that $y(0)$ should be one of these four zeros: write $y^+_4$ for the solution that satisfies $y^+_4(0) = \mu^+$ and $y^-_4$ for the solution that satisfies $y^-_4(0) = \mu^-$; then $- y^+_4$ and $- y^-_4$ are the solutions that satisfy the other two initial conditions. We remark that our $y^+_4$ is the solution $y_4$ upon which the signature-four account in [2016] is based. 

\medbreak 

We may identify $y^+_4$ in terms of its coperiodic Weierstrass $\wp$ function as follows. 

\medbreak 

\begin{theorem} \label{y4}
The elliptic function $y^+_4$ is given by 
$$y^+_4 = \sqrt{\tfrac{1}{2} (1 + \kappa)} \, \Big[ 1 + \frac{4 \kappa}{P - (\frac{4}{3} + 2 \kappa)} \Big]$$
where $P = \wp(\bullet; G_2, G_3)$ is the Weierstrass function with invariants 
$$G_2 = \tfrac{16}{3} (1 + 3 \lambda^2) \; \; {\rm and} \; \; G_3 = \tfrac{64}{27} (1 - 9 \lambda^2).$$
\end{theorem} 

\begin{proof} 
Another application of the Theorem in the Appendix.
\end{proof} 

\medbreak 

As was the case for $\dd$ and its coperiodic Weierstrass function, $y^+_4$ and $P$ share a rectangular period lattice. We shall denote by $(2 \OO, 2 \OO ')$ their fundamental pair of periods for which $\OO$ and $- \, \ii \OO '$ are strictly positive; their precise values will be made explicit in Theorem \ref{ooOO}. We note that the midpoint values of $P$ are 
$$P(\OO) = \tfrac{4}{3}, \, P(\OO + \OO ') = - \tfrac{2}{3} + 2 \lambda, \, P(\OO ') = - \tfrac{2}{3} - 2 \lambda$$
and the corresponding values of $y^+_4$ with $y^+_4 (0) = \mu^+$ are 
$$y^+_4(\OO) = - \mu^+, \; y^+_4(\OO + \OO ') = - \mu^-, \; y^+_4 (\OO ') = \mu^-.$$ 

\medbreak 

The solutions $y^-_4, \, - y^+_4$ and $- y^-_4$ may likewise be expressed in terms of the same Weierstrass function $P$: the expression for $y^-_4$ is obtained from that for $y^+_4$ in Theorem \ref{y4} by substituting $- \kappa$ for $\kappa$ throughout. Alternatively, these other solutions may be obtained from $y^+_4$ by half-period shifts. 

\medbreak 

\begin{theorem} \label{shifts}
$y^+_4 (z + \OO) = - y^+_4(z)$ and $y^+_4 (z + \OO') = y^-_4 (z)$. 
\end{theorem} 

\begin{proof} 
As the differential equation for $y^+_4$ is autonomous, it is also satisfied by the function $f : z \mapsto y^+_4 (z + \OO)$; moreover, $f(0) = y^+_4(\OO) = - \mu^+$. It follows that $f = - y^+_4$; similarly for the $\OO '$ shift. Of course, shifting $y^+_4$ by $\OO + \OO '$ produces $- y^-_4$. 
\end{proof} 

\medbreak 

If $b \in \mathbb{C}$ is not one of the four values $\pm \mu^{\pm}$ then the initial value problem 
$$(y ')^2 = T_4 (y) - (1 - 2 \lambda^2), \; y(0) = b$$ 
has two solutions: as $y^+_4$ is elliptic of order two and $b$ is not one of its double values, there exist (modulo periods) two $a \in \mathbb{C}$ such that $y^+_4 (a) = b$; for each of these, the initial value problem is solved by $z \mapsto y^+_4( z + a )$. 

\medbreak

In particular, the initial value problem 
$$(y ')^2 = T_4 (y) - (1 - 2 \lambda^2), \; y(0) = 0$$ 
has two solutions. In this case, the lack of uniqueness is immediately obvious, for if $y$ is a solution then $- y$ is also a solution. In order to present these solutions explicitly as translates of $y^+_4$ we proceed to locate its zeros; at the same time, we locate its poles. 

\medbreak 

The midpoint values recorded after Theorem \ref {y4} make it clear that $y^+_4$ changes sign along the upper and lower edges of the half-period rectangle. In fact, the sign change occurs exactly half-way along each edge. 

\medbreak 

\begin{theorem} \label{y4zero}
The function $y^+_4$ has a zero at $\tfrac{1}{2} \OO + \OO'$ and a pole at $\tfrac{1}{2} \OO$. 
\end{theorem} 

\begin{proof} 
We import some standard Weierstrass function evaluations from Table XVIII on page 467 of [1910]: if $\wp$ has fundamental half-periods $\OO, \OO '$ and midpoint values $e_1 > e_2 > e_3$ then 
$$\wp(\tfrac{1}{2} \OO + \OO ') = e_1 - \sqrt{(e_1 - e_2)(e_1 - e_3)}$$
and 
$$\wp(\tfrac{1}{2} \OO) = e_1 + \sqrt{(e_1 - e_2)(e_1 - e_3)}.$$
As noted after Theorem \ref{y4}, the Weierstrass function $P$ has midpoint values 
$$e_1 = \tfrac{4}{3}, \, e_2 = - \tfrac{2}{3} + 2 \lambda, \, e_3 = - \tfrac{2}{3} - 2 \lambda$$ 
whence by substitution 
$$P( \tfrac{1}{2} \OO + \OO ') = \tfrac{4}{3} - 2 \kappa$$
and 
$$P(\tfrac{1}{2} \OO) = \tfrac{4}{3} + 2 \kappa.$$ 
The proof is completed by inserting these values in Theorem \ref{y4}. 
\end{proof} 

\medbreak 

It is easily verified that each of these is simple; further, that each zero is congruent to $\pm (\tfrac{1}{2} \OO + \OO ')$ and each pole to $\pm \tfrac{1}{2} \OO$ (modulo periods). 

\medbreak 

\begin{theorem} \label{y0}
A solution to the initial value problem 
$$(y ')^2 = T_4 (y) - (1 - 2 \lambda^2), \; y(0) = 0$$ 
is given by the function 
$$z \mapsto y^+_4 (z + \tfrac{1}{2} \OO + \OO ').$$ 
\end{theorem} 

\begin{proof} 
Immediate from Theorem \ref{y4zero}, the differential equation being autonomous.  
\end{proof} 

\medbreak 

The other solution is of course the negative of this solution; it may also be written as $z \mapsto y^+_4 (z - \tfrac{1}{2} \OO - \OO ')$ in keeping with the fact that an $\OO$ shift converts $y^+_4$ to its negative. 

\medbreak 

\section{Relations} 

\medbreak 

It is apparent from Theorem \ref{p} and Theorem \ref{y4} that the Weierstrass functions $p$ and $P$ are intimately related: their respective invariants make this abundantly clear. Explicitly, they are related as follows.  

\medbreak 

\begin{theorem} \label{pP}
$P(z) = - 4 \, p (2 \, \ii \, z).$
\end{theorem} 

\begin{proof} 
The homogeneity relation 
$$\wp( z; \gamma^4 g_2, \gamma^6 g_3) = \gamma^2 \wp(\gamma z; g_2, g_3)$$
is satisfied by any Weierstrass function. Apply it to the Weierstrass function $\wp = p$ in Theorem \ref{p} with invariants $g_2 = \tfrac{1}{3} (3 \lambda^2 + 1)$ and $\tfrac{1}{27} (9 \lambda^2 - 1)$: the choices $\gamma = \pm 2 \, \ii$ yield $\gamma^4 g_2 = G_2$ and $\gamma^6 g_3 = G_3$ as in Theorem \ref{y4}, so the homogeneity relation above reads $P(z) = - 4 \, p (2 \, \ii \, z)$.
\end{proof} 

\medbreak 

Naturally, this relationship between Weierstrass functions has consequences. The consequent relationship between $\dd$ and $y^+_4$ is readily deduced. 

\medbreak 

\begin{theorem} \label{ddy4}
$$\dd (2 \ii z) = 1 +\kappa \, \frac{y^+_4 (z) - \mu^+}{y^+_4 (z) + \mu^+}.$$
\end{theorem} 

\begin{proof} 
Use Theorem \ref{pP} to eliminate the $\wp$ functions between Theorem \ref{p} and Theorem \ref{y4}. 
\end{proof} 

\medbreak 

\medbreak 

 Another consequence is the fact that we may at once deduce expressions for the fundamental half-periods of $y^+_4$ (equivalently, of $P$) from our expressions for those of $\dd$ (equivalently, of $p$). 

\medbreak 

\begin{theorem} \label{ooOO}
$\oo = - 2 \, \ii \, \OO'$ and $\oo' = 2 \, \ii \, \OO$. 
\end{theorem} 

\begin{proof} 
Theorem \ref{pP} tells us that $z$ is a period of $P$ precisely when $2 \ii  z$ is a period of $p$.  
\end{proof} 

\medbreak 

As regards their period lattices, the passage from $\dd$ to $y^+_4$ involves not only a scaling but also a quarter-rotation; this rotation is visible in the period-ratios. Recall from the close of Section 1 the period-ratio 
$$\frac{\oo '}{\oo} = \ii \, \sqrt2 \; \frac{F(\tfrac{1}{4}, \tfrac{3}{4}; 1 ; \lambda^2)}{F(\tfrac{1}{4}, \tfrac{3}{4}; 1; \kappa^2)}$$
where $\dd$ is constructed using the modulus $\kappa$. Theorem \ref{ooOO} allows us to deduce from this the period-ratio
$$\frac{\OO '}{\OO} = \ii \, \frac{1}{\sqrt2} \; \frac{F(\tfrac{1}{4}, \tfrac{3}{4}; 1 ; \kappa^2)}{F(\tfrac{1}{4}, \tfrac{3}{4}; 1; \lambda^2)}$$
where we recall that $y^+_4$ is defined using the parameter $\lambda$. 

\medbreak 

We close by remarking that there are other ways of passing between $\dd$ and $y^+_4$; here is one such. For the following, in addition to using $\kappa$ as the modulus for $\dd$ we shall also use $\kappa$ (rather than $\lambda$) and as the parameter for $y^+_4$: thus, $y^+_4$ now satisfies the initial value problem 
$$(y ' )^2 = T_4 (y) - (1 - 2 \kappa^2) \; \; {\rm and} \; \; y(0) = \sqrt{(1 + \lambda) / 2}.$$ 
In the one direction, the function 
$$z \mapsto 1 - 2 y^+_4 (\tfrac{1}{\sqrt8} z + \tfrac{1}{2} \OO + \OO ')^2$$ 
is $\dd$ because it satisfies the initial value problem in Theorem \ref{dn2}: it satisfies the differential equation therein by virtue of the differential equation now satisfied by $y^+_4$; it satisfies the initial condition on account of Theorem \ref{y4zero}, wherein the half-periods now involve $\kappa$ as parameter. In the opposite direction, the extraction of a square-root is required: we leave this as an exercise, merely noting that the function $1 - \dd$ has double poles (at points congruent to $\oo '$) and double zeros (at points congruent to $0$), whence it has square-roots that are meromorphic and indeed elliptic; by suitable scaling and shifts, we can arrange that these square-roots are $\pm y^+_4$.  

\medbreak 

\section*{Appendix} 

\medbreak 

For convenience, we explicitly record here the elliptic solutions to a class of first-order initial value problems. The material itself is essentially copied directly from Section 20.6 of [1927]; in effect, we have simply given the integral formulation in [1927] a differential cast. 

\medbreak 

We consider the differential equation $(w')^2 = f(w)$ in which the right-hand side is the quartic 
$$f(w) = a_0 w^4 + 4 a_1 w^3 + 6 a_2 w^2 + 4 a_3 w + a_4$$ 
with quadrinvariant  
$$g_2 = a_0 a_4 - 4 a_1 a_3 + 3 a_2^2$$ 
and cubinvariant  
$$g_3 = a_0 a_2 a_4 + 2 a_1 a_2 a_3 - a_2^3 - a_0 a_3^2 - a_1^2 a_4.$$
We solve this differential equation with initial value $w(0) = w_0$ such that $f(w_0) = 0$; we do so assuming that the zeros of $f$ are simple. 

\medbreak 

\begin{theorem*} \label{App}
The initial value problem 
$$(w')^2 = f(w), \; \; w(0) = w_0$$
has solution given by 
$$w (z) = w_0 + \frac{\frac{1}{4} f'(w_0)}{\wp(z ; g_2, g_3) - \frac{1}{24} f''(w_0)}\, .$$
\end{theorem*}

\begin{proof} 
The Taylor expansion of the quartic $f$ about $w_0$ is 
$$f(w) = A_0 (w - w_0)^4 + 4 A_1 (w - w_0)^3 + 6 A_2 (w - w_0)^2 + 4 A_3 (w - w_0)$$
where 
$$A_0 = a_0, \; A_1 = a_0 w_0 + a_1, \; A_2 = a_0 w_0^2 + 2 a_1 w_0 + a_2, A_3 = a_0 w_0^3 + 3 a_1 w_0^2 + 3 a_2 w_0 + a_3.$$
Put $r = (w - w_0)^{-1}$ to reduce the right-hand side to a cubic: 
$$(r')^2 = (w - w_0)^{-4} (w')^2 = r^4 (A_0 r^{-4} + 4 A_1 r^{-3} + 6 A_2 r^{-2} + 4 A_3 r^{-1})$$
or 
$$(r')^2 = 4 A_3 r^3 + 6 A_2 r^2 + 4 A_1 r + A_0.$$ 
Put $q = A_3 r$ to replace by $4$ the leading coefficient:  
$$(q')^2 = 4 q^3 + 6 A_2 q^2 + 4 A_1 A_3 q + A_0 A_3^2.$$ 
Put $p = q + \frac{1}{2} A_2$ to remove the quadratic term: 
$$(p')^2 = 4 p^3 - (3 A_2^2 - 4 A_1 A_3) p - (2 A_1 A_2 A_3 - A_3^2 - A_0 A_3^2)$$ 
wherein a calculation reveals that  
$$3 A_2^2 - 4 A_1 A_3 = g_2 \; \; {\rm and} \; \; 2 A_1 A_2 A_3 - A_3^2 - A_0 A_3^2 = g_3.$$ 
Thus 
$$(p')^2 = 4 p^3 - g_2 p - g_3$$ 
and the fact that $w = w_0$ at the origin gives $p$ a pole there; so $p = \wp(\bullet; g_2, g_3)$. Reverse the steps, noting that $12 A_2 = f''(w_0)$ and that $4 A_3 = f'(w_0) \neq 0$, to see that $w$ has the advertised dependence on $\wp$. 
\end{proof}

\medbreak

\medbreak

\bigbreak

\begin{center} 
{\small R}{\footnotesize EFERENCES}
\end{center} 
\medbreak 

[1910] H. Hancock, {\it Lectures on the Theory of Elliptic Functions}, John Wiley. 

\medbreak 

[1927] E.T. Whittaker and G.N. Watson, {\it A Course of Modern Analysis}, Fourth Edition, Cambridge University Press. 

\medbreak 

[1944] E.H. Neville, {\it Jacobian Elliptic Functions}, Oxford University Press. 

\medbreak 

[1953] A. Erdelyi (director), {\it Higher Transcendental Functions}, Volume 1, McGraw-Hill. 

\medbreak 

[2014] Li-Chien Shen, {\it On a theory of elliptic functions based on the incomplete integral of the hypergeometric function $_2 F_1 (\frac{1}{4}, \frac{3}{4} ; \frac{1}{2} ; z)$}, Ramanujan Journal {\bf 34} 209-225. 

\medbreak 

[2016] Li-Chien Shen, {\it On Three Differential Equations Associated with Chebyshev Polynomials of Degrees 3, 4 and 6}, Acta Mathematica Sinica, English Series {\bf 33} (1) 21-36. 

\medbreak 

[2020] P.L. Robinson, {\it The elliptic function ${\rm dn}_2$ of Shen}, arXiv 2009.04910. 

\medbreak

\end{document}